\documentclass[a4paper, 11pt]{amsart}
\usepackage{amsmath, amssymb, eucal, amscd, amstext, enumerate, mathrsfs, yfonts, amsfonts, comment}

\newtheorem{thm}{Theorem}[section]
\newtheorem{lem}[thm]{Lemma}

\newtheorem{prop}[thm]{Proposition}

\theoremstyle{definition}

\theoremstyle{remark}

\newtheorem{eg}[thm]{Example}

\numberwithin{equation}{section}

\newcommand{\sa}{{\rm sa}}

\newcommand{\ti}{\tilde}

\newcommand{\RP}{\mathbb{R}_+}

\newcommand{\BC}{\mathbb{C}}

\newcommand{\BN}{\mathbb{N}}
\newcommand{\BR}{\mathbb{R}}

\newcommand{\CF}{\mathcal{F}}

\newcommand{\CI}{\mathcal{I}}

\newcommand{\KS}{\mathfrak{S}}

\newcommand{\KH}{\mathfrak{H}}

\newcommand{\supp}{\mathrm{supp}\ \!}

\begin{document}

\title[]{The positive contractive part of a Noncommutative $L^p$-space is a complete Jordan invariant}

\author{Chi-Wai Leung \and Chi-Keung Ng \and Ngai-Ching Wong}

\address[Chi-Wai Leung]{Department of Mathematics, The Chinese
	University of Hong Kong, Hong Kong.}
\email{cwleung@math.cuhk.edu.hk}

\address[Chi-Keung Ng]{Chern Institute of Mathematics and LPMC, Nankai University, Tianjin 300071, China.}
\email{ckng@nankai.edu.cn}

\address[Ngai-Ching Wong]{Department of Applied Mathematics, National Sun Yat-sen University,  Kaohsiung, 80424, Taiwan.}
\email{wong@math.nsysu.edu.tw}

\date{\today}



\begin{abstract}
Let $1\leq p \leq +\infty$. 
We show that the positive part of the closed unit ball of a non-commmutative $L^p$-space, as a metric space, is a complete Jordan $^*$-invariant for the underlying von Neumann algebra.
\end{abstract}

\maketitle

\section{Introduction}


Given a von Neumann algebra $M$,  celebrated results of R. V. Kadison showed that several partial structures of $M$ can recover the von Neumann algebra
up to Jordan $^*$-isomorphisms.
In particular, each of the following is a complete Jordan $^*$-invariant of $M$: the Banach space structure of the self-adjoint part $M_\sa$ of $M$ (\cite[Theorem 2]{Kad52}),
the ordered vector space structure of $M_\sa$ (\cite[Corollary 5]{Kad52}) and 
the topological convex set structure of the normal state space of $M$
(\cite[Theorem 4.5]{Kad65}).


Let $p\in [1,+\infty]$, and let $L^p(M)$ be the non-commutative $L^p$-space associated to $M$ with
the canonical cone $L^p(M)_+$.
If $M$ is semi-finite, P.-K. Tam showed in \cite{Tam} that the ordered Banach space $(L^p(M)_\sa, L^p(M)_+)$
characterises $M$ up to Jordan $^*$-isomorphisms.
In the case when $M$ is $\sigma$-finite (but not necessarily semi-finite) and $p=2$, the corresponding result follows from 
a result of A. Connes (namely, \cite[Th\'{e}or\`{e}me 3.3]{Connes-Cara}).
On the other hand, extending results of B. Russo (\cite{Russo68})  and F. J.  Yeadon (\cite{Yeadon81}),
D. Sherman showed in \cite{Sher05} that the Banach space  $L^p(M)$ is also a complete Jordan $^*$-invariant for a general von Neumann algebra $M$ when $p\neq 2$.


Along this line, we show in this article that the underlying metric space structure of the positive contractive part
$$
{L^p(M)}^1_{+}:= L^p(M)_+\cap {L^p(M)}^1 \qquad (1\leq p\leq +\infty)
$$ 
of $L ^p(M)$ is also a complete Jordan $^*$-invariant of $M$, where ${L^p(M)}^1$ is the closed unit ball.
More precisely, we will show in Theorem \ref{thm:main} that two arbitrary von Neumann algebras $M$ and $N$ are Jordan $^*$-isomorphic
whenever there exist a bijective isometry $\Phi$ from $L^p(M)_+^1$ onto $L^p(N)_+^1$, i.e.,
$$
\|\Phi(x)-\Phi(y)\| = \| x-y\| \qquad (x,y\in  L^p(M)_+^1).
$$
Notice that the closed unit ball ${L^2(M)}^1$ itself is not a complete Jordan $^*$-invariant (since for any infinite dimensional von Neumann algebra $M$ with a separable predual, one has $L^2(M) \cong \ell^2$), but its positive part is a Jordan $^*$-invariant.

\section{Preliminaries}


Throughout this article, if $E$ is a subset of a normed space $X$ and $\lambda > 0$, we set
$$E^\lambda:= \{x\in E: \|x\| \leq \lambda\}.$$
In the following, we will briefly recall (mainly from \cite{RX03}) notations concerning non-commutative $L^p$-spaces.
Let $M$ be a (complex) von Neumann algebra on a (complex) Hilbert space $\KH$ and $\alpha: \BR \to {\rm Aut}(M)$ be the modular automorphism group.
Then the von Neumann algebra crossed product $\check M :=M\bar\rtimes_\alpha \BR$ is semi-finite and we fix a normal faithful semi-finite trace $\tau$ on $\check M$.
The \emph{measure topology} on $\check M$ (as introduced by E. Nelson in \cite{Nels74}) is given by a neighborhood basis at $0$ of the form
$$U(\epsilon, \delta):= \{x\in \check M: \|xp\|\leq \epsilon \text{ and } \tau(1-p)\leq \delta, \text{ for a projection }p\in \check M\}.$$
The completion, $L_0(\check M, \tau)$, of $\check M$ with respect to this topology is a $^*$-algebra extending the $^*$-algebra structure on $\check M$.


One may identify $L_0(\check M, \tau)$ with a collection of closed and densely defined operators on $L^2(\BR;\KH)$ affiliated with $\check M$.
More precisely, suppose that $T$ is such a closed operator on $L^2(\BR;\KH)$ and that $|T|$ is the absolute value of $T$ with the spectral measure $E_{|T|}$.
Then $T$ corresponds (uniquely) to an element in $L_0(\check M, \tau)$ if and only if $\tau \big(1 - E_{|T|}([0,\lambda])\big) < +\infty$ when $\lambda$ is large enough.
In this case, the $^*$-operation on $L_0(\check M, \tau)$ coincides with the adjoint. 
Moreover, the addition and the multiplication on $L_0(\check M, \tau)$ are the closures of the corresponding operations for densely defined closed operators. 
We denote by $L_0(\check M, \tau)_+$ the set of all positive self-adjoint (but not necessarily bounded) operators in $L_0(\check M, \tau)$.


The dual action $\hat \alpha: \BR \to {\rm Aut}(\check M)$ of $\alpha$ extends to an action on $L_0(\check M, \tau)$ by $^*$-automorphisms.
For any $p\in [1,+\infty]$, we set, as in the literature, 
$$L^p(M):= \{T\in L_0(\check M, \tau): \hat \alpha_s(T) = e^{-s/p}T, \text{ for all }s\in \BR\}.$$
Denote by $L^p(M)_\sa$ the set of all self-adjoint operators in $L^p(M)$ and put
$$L^p(M)_+ := L^p(M)\cap L_0(\check M, \tau)_+.$$
If $T\in L_0(\check M, \tau)$ and $T= u|T|$ is the polar decomposition, then $T\in L^p(M)$ if and only if $u\in M$ and $|T|\in L^p(M)$.


In the case when $p\in (1,+\infty)$, the map that sends $x\in \check M_+$ to $x^p$ extends to a map
$$\Lambda_p: L_0(\check M, \tau)_+ \to L_0(\check M, \tau)_+.$$
For any $T\in L_0(\check M, \tau)_+$, one has $T\in L^p(M)$ if and only if $\Lambda_p(T)\in L^1(M)$.
There is a canonical identification of $M_*$ with $L^1(M)$ that sends the positive part $M_{*,+}$ of $M_*$ onto $L^1(M)_+$, and this induces a Banach space norm $\|\cdot\|_1$ on $L^1(M)$.
The function defined by
\begin{equation}\label{eqt:defn-Lp-norm}
\|T\|_p:= \|\Lambda_p(|T|)\|_1^{1/p}
\end{equation}
is a norm on $L^p(M)$ that turns it into a Banach space.
On the other hand,  one may identify $M$ with $L^\infty(M)$ (as ordered Banach spaces) 
through the canonical inclusion $M\subseteq \check M \subseteq L_0(\check M, \tau)$.


\section{Results and Questions}

\subsection{The case of $p=+\infty$}

\begin{prop}\label{prop:p=infty}
If $\Phi:M_+^1\to N_+^1$ is a bijective isometry, then $\Psi: x \mapsto \Phi(x + \frac{1}{2}) - \frac{1}{2}$ extends to a linear isometry from $M_\sa$ onto $N_\sa$.
\end{prop}

We may then conclude from \cite[Theorem 2]{Kad52} that $x\mapsto \Psi(1)\Psi(x)$ is a Jordan $^*$-isomorphism. 
In order to establish this proposition, we need the following stronger version of the Mazur-Ulam theorem, which was first proved in \cite[Theorem 2]{Mank} (see also \cite[Theorem 14.1]{BL}). 


\begin{lem}\label{lem:Mank}
Let $U$ be a non-empty open connected subset of a  normed space $X$ and $W$ be an open subset of a   normed space $Y$.
Then every isometry from $U$ onto $W$ can be extended uniquely to an affine isometry from $X$ onto $Y$.
\end{lem}


\noindent \emph{Proof of Proposition \ref{prop:p=infty}.}\ 
	Let us first note that for any $x\in M_\sa$, one has $x\in M_+^1$ if and only if $\|x - \frac{1}{2}\| \leq \frac{1}{2}$ (by considering the $C^*$-subalgebra generated by $x$ and $1$).
	Thus, $x\mapsto x - \frac{1}{2}$ is a bijective isometry from $M_+^1$ onto $M_\sa^{\frac{1}{2}}$ and the map $\Psi$ in the statement is a bijective isometry from $M_\sa^{\frac{1}{2}}$ onto $N_\sa^{\frac{1}{2}}$.
	
	If $x\in M_\sa^{\frac{1}{2}}$, then $\|x\| = \frac{1}{2}$ if and only if there exists $x'\in M_\sa^{\frac{1}{2}}$ with $\|x-x'\| = 1$.
	This implies 
$$
\Psi\big(\{x\in M_\sa: \|x\| = 1/2\}\big) = \{y\in N_\sa: \|y\| = 1/2\}.
$$
	Consequently,  $\Psi(0) = 0$ and $\Psi$ will send the interior, $B_M\left(0,\frac{1}{2}\right)$, of $M_\sa^{\frac{1}{2}}$ onto the interior of $N_\sa^{\frac{1}{2}}$.
	By Lemma \ref{lem:Mank}, $\Psi|_{B_M\left(0,\frac{1}{2}\right)}$ extends to a linear isometry $\ti \Psi$ from $M_\sa$ onto $N_\sa$ and the continuity of $\Psi$ tells us that $\ti \Psi|_{M_\sa^{\frac{1}{2}}} = \Psi$. 
\hfill $\Box$\\




\begin{eg}\label{eg:p=infty}
Let $M = \BC\oplus_\infty \BC$. 
The set $M_+^1$ equals the square in $\BR\oplus_\infty \BR$ with vertices $(0,0)$, $(0,1)$, $(1,1)$ and $(1,0)$.
If $\Phi_0: \BR\oplus_\infty \BR \to \BR\oplus_\infty \BR$ is the clockwise rotation by 90 degree about the center $(\frac{1}{2},\frac{1}{2})$, then the restriction $\Phi$ of $\Phi_0$ on $M_+^1$ is a bijective isometry onto $M_+^1$ that sends $(0,0)$ to $(0,1)$. 
Hence, $\Phi$ itself cannot be extended to a linear map.  
However, if $\Psi$ is as defined in Proposition \ref{prop:p=infty}, then 
$\Psi(1,1) = \Phi\left(\frac{3}{2},\frac{3}{2}\right) - \left(\frac{1}{2},\frac{1}{2}\right) = (1,-1)$ and the map 
$$(x,y) \mapsto  \Psi(1,1)\Psi(x,y) = (1, -1)\big(\Phi_0(x+1/2,y+1/2) - (1/2,1/2)\big) 
= (y,x)
$$
is a $^*$-automorphism of $M$. 
\end{eg}

\subsection{The case of $p=1$}


\begin{prop}\label{prop:p=1}
	If there exists a bijective isometry $\Phi$ from $M_{*,+}^1$ onto $N_{*,+}^1$, then $M$ and $N$ are Jordan $^*$-isomorphic.
\end{prop}



Note, first of all, that one cannot use Lemma \ref{lem:Mank} for this case, since the interior of $M_{*,+}^1$ could be an empty set, e.g. when $M= L^\infty([0,1])$. 

For any $\mu \in M_{*,+}$, we denote by $\supp \mu$ the support projection of $\mu$ in $M$. 
Recall that for any  $\mu,\nu\in M_{*,+}$, we have 
\begin{equation}\label{eqt:orth-supp}
\|\mu - \nu\| = \|\mu\| + \|\nu\|\quad\text{if and only if}\quad \supp \mu \cdot \supp \nu = 0. 
\end{equation}
In order to obtain Proposition \ref{prop:p=1}, we need the following lemma.

\begin{lem}\label{lem:3-orth-proj}
If $N$ contains three non-zero orthogonal projections $q_1$, $q_2$ and $q_3$, then the bijective isometry $\Phi$ in Proposition \ref{prop:p=1} will send $0$ to $0$. 
\end{lem}
\begin{proof}
Suppose on the contrary that $\Phi(0) \neq 0$. 
Let us first show that $\supp \Phi(0) = 1$. 
Indeed, if it is not the case, one can find $\mu\in M_{*,+}^1$
such that $\|\Phi(\mu)\| = 1$ and $\supp \Phi(\mu) \leq 1-\supp \Phi(0)$, which, together with \eqref{eqt:orth-supp}, gives the contradiction that 
$$
1\geq \|\mu - 0 \| = \|\Phi(\mu) -\Phi(0)\| = \|\Phi(\mu)\| + \|\Phi(0)\| > 1.
$$
As a result, $\Phi(0)(q_k) > 0$ for $k=1,2,3$. 
We may also assume, without loss of generality, that 
$\Phi(0)(q_1) \leq \|\Phi(0)\|/3$ because 
$$\sum_{k=1}^3\Phi(0)(q_k) \leq \|\Phi(0)\|.$$  
Now, pick any $\nu\in M_{*,+}^1$ with $\|\Phi(\nu)\| = 1$ and $\supp \Phi(\nu) \leq q_1$.
Since $1-2q_1$ is a unitary and $\|\Phi(\nu) - \Phi(0)\| = \|\nu\| \leq 1$, one arrives at the following contradiciton:
$$
1 \geq |(\Phi(\nu) - \Phi(0))(q_1 - (1-q_1))| = |1 - \Phi(0)(q_1) + \Phi(0)(1-q_1)| = 1 +\|\Phi(0)\| - 2\Phi(0)(q_1) > 1.
$$
\end{proof}

Consequently, if $N$ contains three non-zero orthogonal projections, then $\Phi$ induces an isometric bijection from the normal state space of $M$ to that of $N$, and hence, we may conclude that $M$ and $N$ are Jordan $^*$-isomorphic by using \cite[Theorem 3.4]{LNW-tran-prob}. 
For the benefit of the readers, we will instead go through briefly the argument of \cite[Theorem 3.4]{LNW-tran-prob} by recalling the following two lemmas.
These two lemmas are also needed in the case of $p\in (0,+\infty)$ below.

Let us recall that a bijection $\Gamma$ from the lattice of projections in $M$ to that of $N$ is called an \emph{orthoisomorphism} if for any projections $p$ and $q$ in $M$, one has 
$$p\,q=0  \quad \text{if and only if}\quad \Gamma(p) \Gamma(q) = 0. $$


\begin{lem}\label{lem:orth-isom}
(\cite[Lemma 3.1(a)]{LNW-tran-prob})
Suppose that $\Psi$ is a bijection from the normal state space of $M$ to that of $N$, which is 
\emph{biorthogonality preserving} in the sense that  for any normal states $\mu$ and $\nu$ of $M$, one has 
$$
\supp \mu \cdot \supp \nu = 0\quad\text{if and only if}\quad\supp \Psi(\mu) \cdot \supp \Psi(\nu) = 0. 
$$
Then there is an orthoisomorphism $\check \Psi$ from the lattice of projections in $M$ to that of $N$
satisfying $\check \Psi(\supp \mu) = \supp \Psi(\mu)$ for any normal state $\mu$ on 
$M$.
\end{lem}


A second lemma that we need is the following possibly well-known variant of a theorem of H. A. Dye in \cite{Dye} (see e.g. \cite[Lemma 2.2(a)]{LNW-tran-prob}).
Note that an assumption of not having type $I_2$ summand is needed for the original version of Dye's theorem.
However, the variant here has a weaker conclusion and does not need the assumption concerning the absence of type $I_2$ summand. 

\begin{lem}\label{lem:Dye}
If there exists an orthoisomorphism from the lattice of projections in $M$ to that of $N$, then $M$ and $N$ are Jordan $^*$-isomorphic.
\end{lem}

\noindent \emph{Proof of Proposition \ref{prop:p=1}.}\ 
Let us first consider the case when $N$ contains three non-zero orthogonal projections. 
Then by Lemma \ref{lem:3-orth-proj}, the map $\Phi$ restricts to an isometric bijection $\Psi$ from the normal state space of $M$ to that of $N$.  
Moreover, \eqref{eqt:orth-supp} implies that $\Psi$ is biorthogonality preserving. 
Now, the conclusion follows from Lemmas \ref{lem:orth-isom} and \ref{lem:Dye}. 

In the case when $M$ contains three non-zero orthogonal projections, one obtains the same conclusion by considering the bijective isometry $\Phi^{-1}$. 

Suppose that neither $M$ nor $N$ contains three non-zero orthogonal projections. 
Then $M$ and $N$ can only be $\BC$, $\BC\oplus_\infty \BC$ or $M_2(\BC)$. 
Observe that the Hausdorff dimensions of the quasi-state space of $\BC$, $\BC\oplus_\infty \BC$ and $M_2(\BC)$ are $1$, $2$ and $4$ respectively. 
Since a bijective isometry preserves Hausdorff dimensions, we conclude that $M$ and $N$ are $^*$-isomorphic. 
\hfill $\Box$\\


\subsection{A preparation for the case of $p \in (1, +\infty)$}

\begin{prop}\label{prop:ext-Lp}
	Let $p\in (1,+\infty)$.
	Suppose that $M$ and $N$ are two von Neumann algebras such that either $M\neq \BC$ or $N\neq \BC$.
	Then any bijective isometry $\Phi: L^p(M)_+^1\to L^p(N)_+^1$ extends to a linear isometry from $L^p(M)_\sa$ onto $L^p(N)_\sa$.
\end{prop}

Notice that $L^p(M)_\sa$ and $L^p(N)_\sa$ are strictly convex Banach spaces for $p\in (1,+\infty)$ (see e.g., Section 5 of \cite{PX}). 
We recall the following well-known fact concerning strictly convex Banach spaces.


\begin{lem}\label{lem:strict-conv-isom>affine}
	Let $X_1$ and $X_2$ be  Banach spaces such that $X_2$ is strictly convex.  Then every isometry from 
 a convex subset $K$ of $X_1$ into $X_2$ is automatically an affine map.
\end{lem}

In fact, we only need to verify that $f\big((x+ y)/2\big) = \big(f(x) + f(y)\big)/2$,
for any $x\neq y$ in $K$.
By ``shifting'' $K$ and $f$, one may assume that $y =0$ and $f(0) = 0$.
Under this assumption, we have $\|f(z)\| = \|z\|$ ($z\in K$) and 
\begin{equation}\label{eqt:affine}
\|f(x) - f(x/2)\| = \|x -x/2\| = \|f(x)\|/2 = \|f(x)\| - \|x\|/2 = \|f(x)\| - \|f(x/2)\|.
\end{equation}
The strict convexity of $X_2$ gives $f(x) - f(x/2)\in \BR\cdot f(x/2)$.
This, together with the last two equalities in \eqref{eqt:affine}, will produce $f(x) = 2f(x/2)$.


The following lemma is an analogue of \cite[Proposition 3.7]{LNW-tran-prob}. 
Note that we consider in this lemma bijective isometries between the contractive parts instead of those between the norm-one parts of $K_1$ and $K_2$ as in \cite{LNW-tran-prob}.
Moreover, we have a more general setting here. 

\begin{lem}\label{lem:isom-cone>Jord}
	Let $X_1$ and $X_2$ be strictly convex real Banach spaces of dimensions at least two (could be infinite).
	Let $\CF: \RP^{(4)}\to \RP$ be a function satisfying 
	$$\CF(t, t, 0, t) = 0 \qquad (t\in \BR_+).$$ 
	For $k=1,2$, suppose that $K_k\subseteq X_k$ is a closed and proper cone in $X_k$ which is \emph{$\CF$-generating}, in the sense that for any $x\in X_k$, there exist unique elements $x_+, x_-\in K_k$ with 
	$$x = x_+ - x_-  \quad \text{and} \quad \CF(\|x\|, \|x_+\|, \|x_-\|, \|x_+ + x_-\|) =0.$$
	Then there are canonical bijective correspondences amongst the following (given by restrictions):
	\begin{itemize}
		\item the set $\CI$ of real linear isometries from $X_1$ onto $X_2$ that send $K_1$ onto $K_2$.
		\item the set $\CI_B$ of bijective isometries from $K_1^1$ onto $K_2^1$.
		\item the set $\CI_K$ of bijective isometries from $K_1$ onto $K_2$.
	\end{itemize}
\end{lem}
\begin{proof}
If $\Phi\in \CI$, then obviously $\Phi|_{K_1^1}\in \CI_B$.
The assignment $\Phi\mapsto \Phi|_{K_1^1}$ is an injection from $\CI$ to $\CI_B$ because $K_1^1$ linearly spans $X_1$.

Suppose that $\Psi\in \CI_B$.
We put 
$$S_{i}:= \{u\in K_i: \|u\| = 1\} \qquad (i=1,2).$$
The set of extreme points of $K_i^1$ is $S_{i}\cup \{0\}$ (since $X_i$ is strictly convex).
By Lemma \ref{lem:strict-conv-isom>affine}, the map $\Psi$ is affine and hence $\Psi(0)\in S_{2}\cup \{0\}$.
If $\Psi(0)\in S_{2}$, then there is a sequence $\{v_i\}_{i\in \BN}$ in $S_{2}\setminus \{\Psi(0)\}$ with $\|v_i - \Psi(0)\|\to 0$ (as $\dim X_2 > 1$), and hence $\{\Psi^{-1}(v_i)\}_{i\in \BN}$ is a sequence in $S_{1}$ norm-converging to $0$, which is absurd.
Thus, we know that $\Psi(0) = 0$.
Define $\hat \Psi:K_1 \to K_2$ by
\begin{equation}\label{eqt:ext-Psi}
\hat \Psi(0) :=0 \qquad \text{and} \qquad \hat \Psi(u) := \|u\| \Psi\left(u/\|u\|\right) \quad (u\in K_1\setminus \{0\}).
\end{equation}
As $\Psi$ is an affine map sending $0$ to $0$, we see that $\hat \Psi$ extends $\Psi$ and that $\hat \Psi(tu) = t\hat \Psi(u)$ ($u\in K_1, t\in \RP$).
For any $u,v\in K_1$, if $\lambda := \|u\|+\|v\|+1$, then
$$\left\|\hat \Psi(u) - \hat \Psi(v)\right\| = \left\|\lambda \Psi\Big(\frac{u}{\lambda}\Big) - \lambda \Psi\Big(\frac{v}{\lambda}\Big)\right\| = \lambda \Big\|\frac{u}{\lambda} - \frac{v}{\lambda}\Big\| = \|u-v\|.$$
Consequently, $\hat \Psi \in \CI_K$.
The assignment $\Psi\mapsto \hat \Psi$ is clearly injective.

Suppose that $\varphi\in \CI_K$.
Again, Lemma \ref{lem:strict-conv-isom>affine} implies that $\varphi$ is affine, and will send extreme points of $K_1$ to extreme points of $K_2$.
However, as $K_i$ is a proper cone, the only extreme point in $K_i$ is zero ($i=1,2$) and we have $\varphi(0) = 0$.
This means that $\varphi$ is additive and positively homogeneous on $K_1$.
Hence, 
\begin{equation}\label{eqt:norm-sum}
\|\varphi(u) + \varphi(v)\| = \|\varphi(u + v)\| = \|u + v\|\qquad (u,v\in K_1).
\end{equation}
Let us define $\ti \varphi:X_1\to X_2$ by 
$$\ti \varphi(x):= \varphi(x_+) - \varphi(x_-) \qquad (x\in X_1).$$
Since $\CF(t, t, 0, t) = 0$ for all $t\in \BR_+$, one has $x_+ = x$ and $x_- =0$ whenever $x\in K_1$.
Hence, $\ti \varphi$ extends $\varphi$.
Moreover, Relation \eqref{eqt:norm-sum} as well as $\|\ti \varphi(x)\| = \|x\|$ ($x\in X_1$) implies
$$
\CF(\|\ti \varphi(x)\|, \|\varphi(x_+)\|, \|\varphi(x_-)\|, \|\varphi(x_+) + \varphi(x_-)\|) = 
\CF(\|x\|, \|x_+\|, \|x_-\|, \|x_+ + x_-\|) = 0,
$$
and the uniqueness of $\ti \varphi(x)_\pm$ ensures that $\ti \varphi(x)_\pm = \varphi(x_\pm)$ ($x\in X_1$).
Furthermore, for any $x,y\in X_1$, one has 
\begin{align*}
\|\ti \varphi(x) - \ti \varphi(y)\|
&= \|\varphi(x_+) - \varphi(x_-) - \varphi(y_+) + \varphi (y_-)\|
\ =\  \|\varphi(x_+ + y_-) - \varphi(x_- + y_+)\|\\
&= \|(x_+ + y_-) - (x_- + y_+)\|
\ =\ \|x-y\|. 
\end{align*}
Applying the same arguments to 
$$\psi:=\varphi^{-1},$$ 
we will obtain a map $\ti\psi$ from $X_2$ into $X_1$ satisfying
$\ti \psi(z)_\pm = \psi(z_\pm)$ ($z\in X_2$).  
For any $z$ in $X_2$, if we set $x:=\ti\psi(z) \in X_1$, then
$$
\ti\varphi(x) = \varphi(x_+) - \varphi(x_-) = \varphi(\psi(z_+)) - \varphi(\psi(z_-)) = z.
$$
This ensures the surjectivity of $\ti \varphi$.
Hence, $\ti \varphi$ is a bijective isometry sending $0$ to $0$, and the Mazur-Ulam theorem tells us that $\ti \varphi\in \CI$.
It is easy to see that the canonical extension of $\ti \varphi|_{K_1^1}$ to $K_1$ as in \eqref{eqt:ext-Psi} coincides with $\varphi$.	
This completes the proof.
\end{proof}


For any $T\in L^p(M)_\sa$, we denote by $\supp T$ the support projection of $T$, i.e. $\supp T$ is the smallest projection $p$ in $M$ satisfying $T\cdot p  = T$ (or equivalently, $p \cdot T = T$).
Let us recall the following statements concerning $S,T\in L^p(M)_+$ from Fact 1.2 and Fact 1.3 of \cite{RX03}:

\begin{enumerate}[S1).]
\item $\supp \Lambda_p(T) = \supp T$;

\item $S\cdot T = 0$ if and only if $\supp S \cdot \supp T = 0$;

\item if $\supp S \cdot \supp T = 0$, then $\|S+T\|_p^p = \|S-T\|_p^p = \|S\|_p^p + \|T\|_p^p$;

\item if $p\neq 2$ and $\|S+T\|_p^p = \|S-T\|_p^p = \|S\|_p^p + \|T\|_p^p$, then $\supp S \cdot \supp T = 0$.
\end{enumerate}


\noindent \emph{Proof of Proposition \ref{prop:ext-Lp}.}\ 
We note, first of all, that if $M=\BC$, then there are only two extreme points of $L^p(M)_+^1$ and hence the strict convexity of $L^p(N)_\sa$ implies that there are only two extreme points of $L^p(N)_+^1$ (see the argument of Lemma \ref{lem:isom-cone>Jord}), which gives $N = \BC$. 
Therefore, the hypothesis actually implies that the dimensions of both $M_\sa$ and $N_\sa$ are at least two. 

Let us first consider the case when $p\neq 2$, and define a map $\CF_p: \RP^{(4)}\to \RP$ by
$$\CF_p(a,b,c,d) := |a^p-b^p-c^p|+|d^p-b^p-c^p|.$$
For any $T\in L^p(M)_\sa$, we know that $|T|\in L^p(M)_+$.
We denote by $T_+$ and $T_-$ the positive part and the negative part of the self-adjoint operator $T$ respectively. 
As a closed operator, $T_\pm$ is the closure of $\frac{|T| \pm T}{2}$. 
Hence, $T_\pm=\frac{|T| \pm T}{2}$ as elements in $L_0(\check M, \tau)$.
This means that $T_\pm\in L^p(M)_+$ and satisfies $T_+T_- = T_-T_+ = 0$.
Now, we know from (S2) and (S3) that
$$
\CF_p(\|T\|, \|T_+\|, \|T_-\|, \|T_+ + T_-\|) = 0.
$$

Conversely, suppose that $T\in L^p(M)_\sa$ and $R,S\in L^p(M)_+$ such that $T = R -S$ and
$$\CF_p(\|T\|, \|R\|, \|S\|, \|R+S\|) = 0.$$
Then by (S2) and (S4), we have $RS=0$.
Therefore, $(R+S)^2 = (R-S)^2 = T^2$, which implies that $R+S = |T|$ (because $R+S$ is a positive self-adjoint operator; see e.g. \cite[Theorem 12]{Bernau}). 
Consequently,
$$R = T_+ \qquad \text{and} \qquad S = T_-.$$
This means that $L^p(M)_+$ is a $\CF_p$-generating cone of $L^p(M)_\sa$ and we may apply Lemma \ref{lem:isom-cone>Jord} to extend $\Phi$ to a real linear isometry from $L^p(M)_\sa$ onto $L^p(N)_\sa$.

For $p=2$, we know from the proof of Lemma \ref{lem:isom-cone>Jord} that $\Phi(0) = 0$ and hence $\Phi$ restricts to a bijective isometry from the set of norm-one elements in $L^2(M)_+$ onto that of $L^2(N)_+$.
Now, the conclusion follows from \cite[Proposition 3.7]{LNW-tran-prob}.
\hfill $\Box$\\


\subsection{The proof for the case $p\in (1,+\infty)$ and the presentation of the main result}

\begin{thm}\label{thm:main}
Let $M$ and $N$ be two von Neumann algebras and let $p\in [1,+\infty]$.
If there is a bijective isometry $\Phi: L^p(M)_+^1 \to L^p(N)_+^1$, then $M$ and $N$ are Jordan $^*$-isomorphic.
\end{thm}
\begin{proof}
The cases of $p=+\infty$ and $p=1$ are proved in Proposition \ref{prop:p=infty} (together with \cite[Theorem 2]{Kad52}) and Proposition \ref{prop:p=1}, respectively (through the canonical identifications of $L^1(M)$ and $L^\infty(M)$ with $M_*$ and $M$).
Moreover, the case of $p=2$ is already established in \cite[Corollary 3.11]{LNW-tran-prob} 
(due to Proposition \ref{prop:ext-Lp} and \cite[Proposition 3.7]{LNW-tran-prob}).

Now, we consider $p\in (1,+\infty)\setminus \{2\}$.
Without loss of generality, we assume that either $M\neq \BC$ or $N\neq \BC$. 
By Proposition \ref{prop:ext-Lp}, $\Phi$ is an affine map with $\Phi(0)= 0$. 
Furthermore, it follows from Relation \eqref{eqt:defn-Lp-norm} that $\Lambda_p$ induces a bijection from $L^p(M)_+^1$ onto $L^1(M)_+^1$ that sends the norm one part of $L^p(M)_+$  onto the norm one part, $\KS(M)$, of $L^1(M)_+$. 
Hence, $\Phi$ induces a bijection $\hat \Phi: \KS(M) \to \KS(N)$ with $\hat \Phi(A) = \Lambda_p(\Phi(\Lambda_p^{-1}(A))$. 
For any $A,B\in \KS(M)$,  it follows from (S1), (S3) and (S4) that 
$$\supp A \cdot \supp B = 0 \text{ if and only if } \left\|\frac{\Lambda_p^{-1}(A)}{2}+\frac{\Lambda_p^{-1}(B)}{2}\right\|_p^p = \left\|\frac{\Lambda_p^{-1}(A)}{2} - \frac{\Lambda_p^{-1}(B)}{2}\right\|_p^p = 2^{1-p}.$$
As $\Phi$ is both isometric and affine (as well as $\Phi(0) = 0$), we know that $\hat \Phi$ is a biorthogonality preserving bijection between the normal state spaces of $M$ and $N$ (through the identification $L^1(M) = M_*$).
The conclusion now follows from Lemmas \ref{lem:orth-isom} and \ref{lem:Dye}.
\end{proof}

\section*{Acknowledgement}

The authors are supported by National Natural Science Foundation of China (11471168) and Taiwan
	MOST grant (104-2115-M-110-009-MY2).


\bibliographystyle{plain}

\begin{thebibliography}{ZZ}

\bibitem{BL}
Y. Benyamini and J.  Lindenstrauss, \emph{Geometric Nonlinear Functional Analysis}, Amer. Math. Soc. Collo. Publ. \textbf{48}, Amer. Math. Soc. (2000). 

\bibitem{Bernau}
S. J. Bernau, The square root of a positive self-adjoint operator, J. Austral. Math. Soc. \textbf{8} (1968), 17-36.

\bibitem{Connes-Cara}
A. Connes, Caract\'{e}risation des espaces vectoriels ordonn\'{e}s sous-jacents aux alg\`{e}bres de von Neumann, Ann.\ Inst.\ Fourier \textbf{24} (1974), 121-155.

\bibitem{Dye}
H. A. Dye, On the geometry of projections in certain operator algebras, Ann. Math. \textbf{61} (1955), 73-89.

\bibitem{Kad52}
R. V. Kadison, A generalized Schwarz inequality and algebraic invariants for operator algebras, Ann.\ Math.\ \textbf{564} (1952), 494-503.

\bibitem{Kad65}
R. V. Kadison,
Transformations of states in operator theory and dynamics,
Topology \textbf{3} (1965) suppl.\ 2, 177-198.

\bibitem{LNW-tran-prob}
C.-W. Leung, C.-K. Ng and N.-C. Wong, Transition probabilities of normal states determine the Jordan structure of a quantum system, preprint (arXiv:1510.01487).

\bibitem{Mank}
P. Mankiewicz, On extension of isometries in normed linear spaces, Bull. Acad. Polon. Sci. Sér. Sci. Math. Astronom. Phys. \textbf{20} (1972), 367-371.

\bibitem{Nels74}
E. Nelson, Notes on non-commutative integration, J. Funct. Anal. \textbf{15} (1974), 103-116.

\bibitem{PX}
G. Pisier, Q. Xu, Non-commutative Lp-spaces, \emph{Handbook of the geometry of Banach spaces, Vol. 2}, North-Holland, Amsterdam (2003), 1459-1517. 

\bibitem{RX03}
Y. Raynaud and Q. Xu, On subspaces of noncommutative $L_p$-spaces, J.\ Funct.\ Anal.\ \textbf{203} (2003), 149-196.

\bibitem{Russo68}
B. Russo,
Isometrics of $L^p$-spaces associated with finite von Neumann algebras,
Bull.\ Amer.\ Math.\ Soc.\ \textbf{74} (1968), 228-232.

\bibitem{Sher05}
D. Sherman, Noncommutative $L^p$-structure encodes exactly Jordan structure, J.\ Funct.\ Anal.\ \textbf{211} (2005), 150-166.

\bibitem{Take}
M. Takesaki, \emph{Theory of operator algebras I}, Springer-Verlag (1979).

\bibitem{Tam}
P.-K. Tam, Isometries of $L_p$-spaces associated with semifinite von Neumann algebras, Trans. Amer. Math. Soc. \textbf{254} (1979), 339-354.

\bibitem{Yeadon81}
F. J. Yeadon,
Isometries of noncommutative $L^p$-spaces.
Math.\ Proc.\ Cambridge Philos.\ Soc.\ \textbf{90} (1981), no. 1, 41-50.
\end{thebibliography}

\end{document}